\documentclass[psamsfonts]{amsart}

\usepackage{amssymb,amsfonts}
\usepackage[all,arc]{xy}
\usepackage{enumerate}
\usepackage{mathrsfs}
\usepackage[utf8]{inputenc}
\usepackage{hyperref}
\usepackage{mathtools}
\usepackage{tikz}
\usepackage{enumitem}
\newtheorem{thm}{Theorem}[section]
\newtheorem{theorem}{Theorem}[section]

\newtheorem{lem}[thm]{Lemma}

\newtheorem{claim}[thm]{Claim}

\theoremstyle{definition}
\newtheorem{definition}[thm]{Definition}

\theoremstyle{remark}
\newtheorem{remark}[thm]{Remark}

\newtheorem{notation}[thm]{Notation}

\makeatletter
\let\c@equation\c@thm
\makeatother
\numberwithin{equation}{section}

\title[Monochromatic Sumset Without large cardinals]{Monochromatic Sumset Without the use of large cardinals}

\author{Jing Zhang}

\newcommand{\Addresses}{{
  \bigskip
  \footnotesize

 \textsc{Department of Mathematics,\\ Bar-Ilan University, Ramat-Gan 5290002, Israel}\par\nopagebreak
  \textit{E-mail}: \texttt{jingzhan@alumni.cmu.edu}
}}

\begin{document}

\begin{abstract}
We show in this note that in the forcing extension by $\mathrm{Add}(\omega,\beth_{\omega})$, the following Ramsey property holds: for any $r\in \omega$ and any $f: \mathbb{R}\to r$, there exists an infinite $X\subset \mathbb{R}$ such that $X+X$ is monochromatic under $f$. We also show the Ramsey statement above is true in $\mathrm{ZFC}$ when $r=2$. This answers two questions from \cite{6authors}.
\end{abstract}

\maketitle
\let\thefootnote\relax\footnotetext{

The work was done when I was a graduate student at Carnegie Mellon University supported in part by the US tax payers. Part of the revision was done when I am a post doctoral fellow in Bar-Ilan University, supported by the Foreign Postdoctoral Fellowship Program of the Israel Academy of Sciences and Humanities and by the Israel Science Foundation (grant agreement 2066/18).
}

\section{Introduction}

\begin{definition}
Let $(A, +)$ be an additive structure and $\kappa,r$ be cardinals. Let $A\to^+ (\kappa)_{r}$ abbreviate the statement: for any $f: A\to r$, there exists $X\subset A$ with $|X|=\kappa$ such that $X+X=_{def} \{a+b: a,b\in X\}$ is monochromatic under $f$.
\end{definition}

There have been recent developments on additive partition relations for real numbers. Hindman, Leader and Strauss \cite{MR3696151} showed that if $2^{\omega}<\aleph_{\omega}$ then there exists some $r\in \omega$ such that $\mathbb{R}\not\to^+ (\aleph_0)_r$. On the other hand, Komjáth, Leader, Russell, Shelah, Soukup and Vidnyánszky \cite{6authors} showed that relative to the existence of an $\omega_1$-Erd\H{o}s cardinal, it is consistent that for any $r\in \omega$, $\mathbb{R}\to^+(\aleph_0)_{r}$.
These results are optimal in a sense as there exist the following restrictions: 
\begin{enumerate}
\item Komjáth \cite{MR3511943} and, independently, Soukup and Weiss \cite{SoukupWeiss} showed that $\mathbb{R}\not\to^+(\aleph_1)_2$;
\item Soukup and Vidnyánszky showed there exists a finite coloring $f$ on $\mathbb{R}$ such that no infinite $X\subset \mathbb{R}$ satisfies that $\underbrace{X+\cdots + X}_{k}$ is monochromatic for $k\geq 3$.
\end{enumerate}

It should be emphasized that the difficulty comes from the fact that repetitions are allowed. If we only want some infinite $X\subset \mathbb{R}$ such that $X\oplus X =\{a+b: a\neq b\in X\}$ is monochromatic, then the classical theorem of Ramsey implies this already. In fact, Hindman's finite-sum theorem is a much stronger Ramsey-type statement: for any finite coloring of $\mathbb{N}$, there exists some infinite $X\subset \mathbb{N}$ such that $FS(X)=_{def}\{\Sigma_{0\leq i <k} a_i: \{a_0,a_1,\cdots, a_{k-1}\} \in [X]^{<\omega}\}$ is monochromatic. However, if repeated sums are allowed, things turn towards the other direction: Hindman \cite{MR541340} showed $\mathbb{N} \not\to^+ (\aleph_0)_3$ and Owings asked (and it is still open) that if $\mathbb{N} \not\to^+ (\aleph_0)_2$ is true. Interestingly, Fernández-Bretón and Rinot \cite{MR3710649} showed that the uncountable analogs of Hindman's theorem must necessarily fail in a strong way.

There have been other results in the similar vein showing that every finite partition of some large uncountable structure, there must be some countable substructure that is ``large'' and ``well-behaved''. The exact meaning of ``large'' and ``well-behaved'' depends on specific situations. 
Sometimes it is not possible to get a homogeneous substructure satisfying some ``largeness'' requirement, but there are ways to find such a substructure as close to homogeneous as possible. For example, recent work of Raghavan and Todorcevic \cite{RaghavanTodorcevic} shows the existence of certain large cardinals in the universe \emph{directly implies} that for any finite partition of $[\mathbb{R}]^2$, there exists a homeomorphic copy of $\mathbb{Q}$, say $X\subset \mathbb{R}$, such that $[X]^2$ intersects at most 2 elements in the partition.

The following questions among others were asked by the authors of \cite{6authors}.
\begin{enumerate}
\item Is the use of large cardinals necessary to establish the consistency $\mathbb{R}\to^+(\aleph_0)_r$ for all $r\in \omega$?
\item Is $\mathbb{R}\to^+(\aleph_0)_2$ true in $\mathrm{ZFC}$?
\end{enumerate}

We answer the first question negatively and the second positively.

\begin{theorem}\label{main}
\begin{enumerate}
\item \label{part1} In the forcing extension by $\mathrm{Add}(\omega, \beth_\omega)$, $\mathbb{R}\to^+ (\aleph_0)_{r}$ for any $r\in \omega$.
\item \label{part2} $\mathbb{R}\to^+(\aleph_0)_2$.
\end{enumerate}
\end{theorem}

\begin{remark}
The continuum in the model of \cite{6authors} is an $\aleph$-fixed point, which is very large. Over a ground model of $\mathrm{GCH}$, Theorem \ref{main} suggests that the most natural way to eliminate the obstacles from cardinal arithmetic works since by a result of Hindman, Leader and Strauss \cite{MR3696151}, if $\mathbb{R}\to^+(\aleph_0)_r$ for all $r<\omega$, then $2^\omega\geq \aleph_{\omega+1}$. 
\end{remark}

\begin{notation}
We will identify $(\mathbb{R}, +)$, as a vector space over $\mathbb{Q}$, with $\bigoplus_{i<2^\omega}\mathbb{Q}$. The latter is the direct sum of $2^\omega$ copies of $(\mathbb{Q},+)$. More concretely, any $s\in \bigoplus_{i<2^\omega}(\mathbb{Q},+)$ is a finitely supported function whose range is contained in $\mathbb{Q}$. The addition on the direct sum is defined coordinate-wise. Similarly for some cardinal $\kappa$, $\bigoplus_{i<\kappa} \mathbb{N}$ is the direct sum of $\kappa$ copies of $(\mathbb{N}, +)$. It is easy to see that if $\kappa\leq 2^\omega$, $\bigoplus_{i<\kappa} \mathbb{N}$ is an additive substructure of $\mathbb{R}$.
\end{notation}

\section{The proof of Theorem \ref{main}}\label{mainresult}

\begin{definition}[\cite{LeaderRussell},\cite{MR3696151},\cite{6authors}]\label{strings}
For any $r\geq 2$, define a sequence of finite strings of natural numbers $\langle s_{r,l}: l\leq r\rangle$ such that for each $l\leq r$, $|s_l|=r+l$ and $s_l(k) = \begin{cases}
2 &  \text{if }k<2l \\
4 &  \text{otherwise}.
\end{cases}
$
In other words, each $s_l$ is formed by $2l$ many $2's$ followed by $r-l$ many $4's$.
\end{definition}

\begin{remark}\label{remarkstrings}
Since $r$ will be fixed in the proof of Theorem \ref{main}, we will suppress the subscripts and write $\langle s_l: l\leq r\rangle$ for $\langle s_{r,l}: l\leq r\rangle$.
\end{remark}

\begin{definition}[The star operation, see \cite{LeaderRussell},\cite{6authors}]
Let $K$ be either $\mathbb{N}$ or $\mathbb{Q}$.
For $k\in \omega$, $s\in (K-\{0\})^k$ and a finite subset of ordinals $a=\{\xi_i: i< k\}_< \subset \lambda$, let $s*a$ denote the function from $\lambda$ to $K$ supported on $a$ that sends $\xi_i$ to $s(i)$.

\end{definition}

We first prove part (\ref{part2}), which is simpler but contains some ideas which will be used later in the proof of part (\ref{part1}).

\begin{proof}[Proof of part (\ref{part2})]

First we use the following partition theorem, which goes back to Dushnik-Miller (see Theorem 11.3 in \cite{MR795592}). 
\begin{claim}\label{DM}
For any $n,k\in \omega$, any $f: [\omega_1]^n\to k$, there exists $A\subset \omega_1$ of order type $\omega+1$, such that $f\restriction A$ is constant.
\end{claim}

\begin{proof}[Proof of the claim]
Given $n, k, f$ as in the claim, find a countable elementary submodel $N\prec H(\omega_2)$ containing $f$. We may assume $n\geq 2$.
Let $\delta=_{def} N\cap \omega_1$. Recursively, we can build an infinite set $A\subset \delta$ such that for any increasing $\{x_i: i<n\}\subset A$, $f(x_0,\cdots, x_{n-1})=f(x_0,\cdots, x_{n-2}, \delta)$. To see how this is done, suppose at some stage we have built a finite set $A'$ satisfying the requirement above and we demonstrate how to augment the set by one more element. For each $\bar{y}\in [A']^{n-1}$, let $i_{\bar{y}}=f(\bar{y}\cup \{\delta\})\in k$. Clearly $\langle i_{\bar{y}}: \bar{y}\in [A']^{n-1}\rangle \in N$. For each $\bar{y}\in [A']^{n-1}$, consider $B_{\bar{y}}=\{\alpha<\omega_1: f(\bar{y}\cup \{\alpha\})=i_{\bar{y}}\}$. Then $\bigcap_{\bar{y}\in [A']^{n-1}} B_{\bar{y}}$ is unbounded in $\delta$ since $\delta\in \bigcap_{\bar{y}\in [A']^{n-1}} B_{\bar{y}}$ and $\bigcap_{\bar{y}\in [A']^{n-1}} B_{\bar{y}}\in N$. Picking any element from $(\delta\cap \bigcap_{\bar{y}\in [A']^{n-1}} B_{\bar{y}} )- (\sup A' +1)$ and adding it to $A'$ will maintain the satisfaction of the requirement.

Finally, consider $g: [A]^{n-1} \to k$ defined by $g(\bar{a})=f(\bar{a}\cup \{\delta\})$. Apply the Ramsey theorem, we get an infinite set $A^*\subset A$ such that $g$ is constant on $[A^*]^{n-1}$ with color $j<k$. It is clear that $f\restriction  [A^*\cup \{\delta\}]^n\equiv j$.
\end{proof}

Returning back to the proof of part (\ref{part2}), we actually prove a stronger statement: $\bigoplus_{i<\omega_1}\mathbb{N}\to^+(\aleph_0)_2$. Recall $\langle s_i: i\leq 2\rangle=\langle s_{2,i}: i\leq 2\rangle$ as in Definition \ref{strings} and Remark \ref{remarkstrings}. Let $f: \bigoplus_{i<\omega_1}\mathbb{N} \to 2$ be given. Let $d_i (\bar{a}) = f(s_i* \bar{a})$ be defined for $i<3$. In particular, the domain of $d_i$ is $[\omega_1]^{i+2}$ for $i<3$. Apply Claim \ref{DM} to get $A=\{\alpha_j: j\leq \omega\}\in [\omega_1]^{\omega+1}$ such that $d_i\restriction [A]^{i+2}\equiv \rho_i<2$ for all $i<3$. By the Pigeon Hole principle we have the following cases and we will define $X=\{x_i: i\in \omega\}$ for each case. 
\begin{enumerate}
\item $\rho_0=\rho_1=\rho$. Let $x_i=\dfrac{1}{2} s_0 * (\alpha_i, \alpha_\omega)$. Then $f(2x_i)=f(s_0*(\alpha_i, \alpha_\omega))=d_0(\alpha_i, \alpha_\omega)=\rho_0=\rho$. For any $i<j\in \omega$, $f(x_i+x_j)=f(s_1*(\alpha_i,\alpha_j,\alpha_\omega))=d_1(\alpha_i,\alpha_j,\alpha_\omega)=\rho_1=\rho$.
\item $\rho_0=\rho_2=\rho$. Let $x_i=\dfrac{1}{2} s_0 * (\alpha_{2i},\alpha_{2i+1})$. Then $f(2x_i)=f(s_0*(\alpha_{2i}, \alpha_{2i+1}))=d_0(\alpha_{2i}, \alpha_{2i+1})=\rho_0=\rho$. For any $i<j\in \omega$, $f(x_i+x_j)=f(s_2*(\alpha_{2i},\alpha_{2i+1},\alpha_{2j},\alpha_{2j+1}))=d_2(\alpha_{2i},\alpha_{2i+1},\alpha_{2j},\alpha_{2j+1})=\rho_2=\rho$.
\item $\rho_2=\rho_1=\rho$. Let $x_i=\dfrac{1}{2} s_0 * (\alpha_{0},\alpha_{1}, \alpha_{i+2})$. Then $f(2x_i)=f(s_0*(\alpha_{0},\alpha_{1}, \alpha_{i+2}))=d_0(\alpha_{0},\alpha_{1}, \alpha_{i+2})=\rho_0=\rho$. For any $i<j\in \omega$, $f(x_i+x_j)=f(s_2*(\alpha_{0},\alpha_{1},\alpha_{i+2},\alpha_{j+2}))=d_2(\alpha_{0},\alpha_{1},\alpha_{i+2},\alpha_{j+2})=\rho_2=\rho$.
\end{enumerate}
\end{proof}

Clearly the proof above does not generalize to the case when $r=3$ since $2^\omega\not\to (\omega+2)^3_2$. A more fundamental restriction is that by a result of Hindman, Leader and Strauss \cite{MR3696151}, there exists some $r\in \omega$ such that $\bigoplus_{i<\omega_1}\mathbb{N}\not\to^+ (\aleph_0)_{r}$.

We dedicate the rest of the article to proving part (\ref{part1}). Let $\lambda=\beth_{\omega}$ and $\mathbb{P}=\mathrm{Add}(\omega,\lambda)$. In fact, we show that in $V^{\mathbb{P}}$, $\bigoplus_{i<\lambda} \mathbb{N}\to^+ (\aleph_0)_{r}$ for any $r\in \omega$.

\begin{definition}
Suppose $W,W'\subset \lambda$ are such that $type(W)=type(W')$. Let $h_{W,W'}: W\to W'$ be the unique order isomorphism. For $A, A'\subset \lambda$ with $type(A)=type(A')$, $h_{A, A'}$ naturally induces a map from $\mathbb{P}\restriction A$ to $\mathbb{P}\restriction A'$ where any $p\in \mathbb{P}\restriction A$ is mapped to $p'\in \mathbb{P}\restriction A'$ such that $dom(p')=h_{A,A'} (dom(p))$ and $p'(j)=p(h_{A, A'}^{-1}(j))$. We will abuse the notation by using $h_{A, A'}$ to denote the induced map from $\mathbb{P}\restriction A$ to $\mathbb{P}\restriction A'$. This can be easily inferred from the context.
\end{definition}

We will use the following combinatorial lemma due to Todor\v{c}evi\'{c} in \cite{MR874786}.

\begin{lem}[The higher dimensional $\Delta$-system lemma]\label{Cleanup}
Fix $r, d\in \omega$.
Let $\langle \dot{d}_i: [\lambda]^{i}\to r | i\leq  d\} $ be a sequence of $\mathbb{P}$-names for colorings. Then there exists $E\subset \lambda$ of order type $\omega_1$ and $W: [E]^{\leq d}\to [\lambda]^{\leq \aleph_0}$ such that 
\begin{enumerate}[label=\textbf{CL.\arabic*}]
\item \label{CL1} For all $u\in [E]^{\leq d}$, $u\subset W(u)$ and $\mathbb{P}\restriction W(u)$ contains a maximal antichain deciding the value of $\dot{d}_{|u|}(u)$.
\item \label{CL2} For any $u,v\in [E]^{\leq d}$ such that $|u|=|v|$, $type(W(u))=type(W(v))$, $h_{W(u), W(v)}(u)=v$ and for any $p\in \mathbb{P}\restriction W(u)$, for any $n<r$, $p\Vdash \dot{d}_{|u|}(u)=n \Leftrightarrow h_{W(u), W(v)}(p)\Vdash \dot{d}_{|v|}(v)=n$.
\item \label{CL3}  For any $u, v \in [E]^{\leq d}$, $W(u)\cap W(v)=W(u\cap v)$.
\item \label{CL4}  For any $u_1\subset u_2, u'_1\subset u'_2 $ where $u_2, u_2'\in [E]^{\leq d}$, if $(u_2, u_1,<)\simeq (u_2', u_1', <),$ then $h_{W(u_2), W(u_2')}\restriction W(u_1) = h_{W(u_1), W(u_1')}$.

\end{enumerate}
\end{lem}

\begin{remark}\label{stronger}
The concept of a \emph{double $\Delta$-system} appeared in Section 2 and 3 of \cite{MR874786}, which can be regarded as a blueprint for the 2-dimensional (that is, when $d=2$) $\Delta$-system lemma as stated in Lemma \ref{Cleanup}. Different variations of Lemma \ref{Cleanup} appeared in \cite{MR1031773}, Lemma 4.1 of \cite{MR1218224}, Claim 7.2.a of \cite{MR2520110} and the appendix of \cite{Zhang1}. 
The fact that $\lambda\to (\aleph_1)^{2d}_{2^\omega}$ suffices to get the conclusion of Lemma \ref{Cleanup}. The interested readers are directed to the proofs in Claim 7.2.a of \cite{MR2520110} (for \ref{CL1},\ref{CL2},\ref{CL3}) and in the appendix of \cite{Zhang1} (for \ref{CL4}).
\end{remark}

\begin{remark}
The referee pointed out that Lemma \ref{Cleanup} can also be obtained using Theorem 1 from Baumgartner \cite{BaumgartnerCanonical}, which is a generalization of the Erd\H{o}s-Rado theorem \cite{doi:10.1112/jlms/s1-25.4.249}.
\end{remark}

Let $G\subset \mathbb{P}$ be generic over $V$. In $V[G]$, suppose $f: \bigoplus_{i<\lambda} \mathbb{N}\to r$ is the given coloring. Define $d_i: [\lambda]^{r+i}\to r$ such that $d_i(\bar{a})=f(s_i*\bar{a})$ for $i\leq r$. Let $\dot{d}_i$ for $i\leq r$ be the corresponding names.

Back in $V$, apply Lemma \ref{Cleanup} to $d=2r$ and $\langle \dot{d}_i: i\leq r\rangle$, and find the desired $E$ and $W$ (strictly speaking, we should apply to the sequence $\langle \dot{d}'_{i+r}: i\leq r\rangle$ where $\dot{d}'_{i+r}=\dot{d}_i$ for $i\leq r$). Enumerate $E$ increasingly as $\{e_i: i\in \omega_1\}$. Let $A_i=\{e_{\omega\cdot i+j}: 1\leq j\leq \omega\}$ for each $i<r$. For each $i<r, j\leq \omega$, let $\alpha^i_j=e_{\omega\cdot i+(1+j)}$.

\begin{definition}
For any $l\leq r$ and any tuple $\bar{s}\in \Pi_{i<l} [A_i]^2 \times \Pi_{i\geq l, i<r} A_i$, we naturally identify $\bar{s}$ as an $(r+l)$-tuple. To be more concrete, we take 2 elements from each of the first $l$ sets ordered naturally and 1 element from each of the remaining sets.
\begin{enumerate}  
\item $\bar{s}$ is \emph{$l$-canonical} if $\bar{s}$ is of the form $$(\alpha^0_{i_0}, \alpha^0_{i_0'}, \cdots, \alpha^{l-1}_{i_{l-1}}, \alpha^{l-1}_{i_{l-1}'}, \alpha^l_{i_l}, \alpha^{l+1}_{i_{l+1}}, \cdots, \alpha^{r-1}_{i_{r-1}})$$ such that for any $k<l$, $i_k<i_k'\leq \omega$ and $\max\{i_{m}: m<r, i_m<\omega\}<i_k'$ for any $k<l$. If, in addition, we are given a sequence $\langle D_i\subset A_i: i<r\rangle$, then we say $\bar{s}$ is \emph{from} $\langle D_i: i<r\rangle$ if $\bar{s}\in \Pi_{i<l} [D_i]^2 \times \Pi_{i\geq l, i<r} D_i$.
\item We call $\bar{i}=\langle i_k: k<r\rangle$ the \emph{index} of an $l$-canonical tuple $\bar{s}$. We say that $\bar{s}$ is \emph{index-strictly-increasing} if whenever $k<k'<r$, $i_k\leq i_{k'}$ and if $i_k\in \omega$, then $i_k<i_{k'}$.

\item For any two ordinals $\alpha, \alpha'$, let $\bar{s}_{\alpha\to \alpha'}$ denote the tuple obtained by replacing the occurrence of $\alpha$ in $\bar{s}$ by $\alpha'$. Similarly for any two sequences of ordinals $\bar{\alpha}, \bar{\alpha}'$ of the same length, $\bar{s}_{\bar{\alpha}\to \bar{\alpha}'}$ denotes the tuple obtained by replacing the occurence of $\alpha_i$ in $\bar{s}$ by $\alpha'_i$ for each $i<|\bar{\alpha}|$. This notation is used to make the statements of Claim \ref{ModularOperation} and Claim \ref{shrinking} precise.

\end{enumerate}
\end{definition}

\begin{notation}
Many times in what follows, we confuse a tuple with the set that consists of elements from the tuple, namely $\bar{s}=\langle s_i: i<n\rangle$ is identified with $\{s_i: i<n\}$. It can be mostly inferred from the context, for example $W(\bar{s})=W(\{s_i: i<n\})$ and $W(\bar{s}\cap \bar{t})=W(\{s_i: i<n\}\cap \{t_j: j<m\})$ where $\bar{t}=\langle t_j: j<m\rangle$.
\end{notation}

\begin{claim}\label{ModularOperation}
In $V[G]$, for any $j<r$ and for any finite $B_i\subset A_i$ with $a^i_\omega\in B_i$ for $i<r$, there exists arbitrarily large $\alpha\in A_j\backslash \{\alpha^j_\omega\}$ such that $\alpha>B_j\backslash \{\alpha^j_\omega\}$ and the following is true: 
for any $l\leq r$, any $l$-canonical tuple $\bar{s}$ from $\langle B_i: i<r\rangle$ containing $\alpha^j_\omega$, $d_l(\bar{s}')=d_l(\bar{s})$ where $\bar{s}'=\bar{s}_{\alpha^j_\omega \to \alpha}$.

\end{claim}
\begin{proof}
Fix $j<r$. Work in $V$. 
For any given $p\in \mathbb{P}$ and $\gamma\in A_j\backslash\{\alpha^j_\omega\}$, we want to find $p'\leq p$ and $\alpha>\max\{\gamma, \max B_j\backslash \{\alpha^j_\omega\}\}$ in $A_j\backslash \{\alpha^j_\omega\}$ such that $p'$ forces the conclusion above is true for this $\alpha$. This clearly suffices by the density argument.

Given $p\in \mathbb{P}$, extending it if necessary, we may assume that for each $l\leq r$ and each $l$-canonical tuple $\bar{s}$ from $\langle B_i: i<r\rangle$, $p\restriction W(\bar{s})$ decides the value of $\dot{d}_{l} (\bar{s})$. Find $\alpha\in A_j\backslash \{\alpha^j_\omega\}$ large enough such that 
\begin{itemize}
\item $\alpha>\max\{\max B_j\backslash \{\alpha^j_\omega\},\gamma\}$
\item $dom(p)\cap (W(u\cup \{\alpha\})-W(u))=\emptyset$ for all $u\in  [\bigcup_{i<r}B_i]^{\leq 2r-1}$.
\end{itemize}
This is possible since $dom(p)$ is finite and for any fix $u\in  [\bigcup_{i<r}B_i]^{\leq 2r-1}$, $W(u\cup \{\alpha\}) \cap W(u\cup \{\alpha'\}) = W(u)$ for any $\alpha\neq \alpha'> \max u +1$.

Define $p'=p\cup \bigcup_{l\leq r} \{h_{W(\bar{s}), W(\bar{s}')}(p\restriction W(\bar{s})): \bar{s} \text{ is an } l\text{-canonical tuple from }\langle B_i: i<r\rangle, \alpha^j_\omega\in \bar{s}, \text{ and $\bar{s}'=\bar{s}_{\alpha^j_\omega\to \alpha}$} \}$. We claim that $p'$ is the desired condition. To verify this, it suffices to show the following: 

\begin{enumerate}
\item $p'$ is a condition. We do this by showing for $p$ is compatible with $h_{W(\bar{s}), W(\bar{s}')}(p\restriction W(\bar{s}))$ and $h_{W(\bar{s}), W(\bar{s}')}(p\restriction W(\bar{s}))$ is compatible with $h_{W(\bar{t}), W(\bar{t}')}(p\restriction W(\bar{t}))$ for each $\bar{s},\bar{t}$ as above.
	\begin{itemize}
	\item Fix $\bar{s}$. To see $p$ is compatible with $p^*=_{def}h_{W(\bar{s}), W(\bar{s}')}(p\restriction W(\bar{s}))$, notice that $dom(p)\cap dom(p^*)\subset dom(p)\cap W(\bar{s}') \subset W(\bar{s}'-\{\alpha\})$ by the choice of $\alpha$. By \ref{CL4}, $h_{W(\bar{s}), W(\bar{s}')}\restriction W(\bar{s}-\{\alpha^j_\omega\})$ is the identity function on $W(\bar{s}-\{\alpha^j_\omega\})$ since $(\bar{s}, \bar{s}-\{\alpha^j_\omega\}, <)\simeq (\bar{s}', \bar{s}'-\{\alpha\}, <)$ and $\bar{s}-\{\alpha^j_\omega\}=\bar{s}'-\{\alpha\}$. Hence $p^*\restriction W(\bar{s}'-\{\alpha\})=p\restriction W(\bar{s}-\{\alpha^j_\omega\})=p\restriction W(\bar{s}'-\{\alpha\})$. 
	\item Fix $\bar{s}, \bar{t}$ as above. Let $q_0=_{def}h_{W(\bar{s}), W(\bar{s}')}(p\restriction W(\bar{s})), q_1=_{def}h_{W(\bar{t}), W(\bar{t}')}(p\restriction W(\bar{t}))$. Notice that $dom(q_0)\cap dom(q_1)\subset W(\bar{s}')\cap W(\bar{t}')=W(\bar{s}'\cap \bar{t}')=W((\bar{s}\cap \bar{t})_{\alpha^j_\omega\to \alpha})$. Observe that $(\bar{s}, \bar{s}\cap \bar{t}, <)\simeq (\bar{s}', \bar{s}'\cap \bar{t}', <)$ and $(\bar{t}, \bar{s}\cap \bar{t}, <)\simeq (\bar{t}', \bar{s}'\cap \bar{t}', <)$. By \ref{CL4}, we have $h_{W(\bar{s}), W(\bar{s}')} (W(\bar{s}\cap \bar{t}))=W(\bar{s}'\cap \bar{t}')$ and $h_{W(\bar{t}), W(\bar{t}')} (W(\bar{s}\cap \bar{t}))=W(\bar{s}'\cap \bar{t}')$.
	Hence $q_0\restriction W((\bar{s}\cap \bar{t})_{\alpha^j_\omega\to \alpha})=h_{W(\bar{s}), W(\bar{s}')} (p\restriction W(\bar{s}\cap \bar{t}))=h_{W(\bar{s}\cap \bar{t}), W(\bar{s}'\cap \bar{t}')}(p\restriction W(\bar{s}\cap \bar{t}))$ and $q_1\restriction W((\bar{s}\cap \bar{t})_{\alpha^j_\omega\to \alpha})=h_{W(\bar{t}), W(\bar{t}')} (p\restriction W(\bar{s}\cap \bar{t}))=h_{W(\bar{s}\cap \bar{t}), W(\bar{s}'\cap \bar{t}')}(p\restriction W(\bar{s}\cap \bar{t}))$. Since $q_0$ and $q_1$ agree on their common domain, it follows that they are compatible.
	\end{itemize}

\item $p'$ forces $\dot{d}_l(\bar{s})=\dot{d}_l(\bar{s}')$ for any $l$-canonical tuple $\bar{s}$ from $\langle B_i: i<r\rangle $ containing $\alpha^j_\omega$ where $\bar{s}'=\bar{s}_{\alpha^j_\omega \to \alpha}$ for any $l\leq r$. Fix $l$ and $\bar{s}$. By the initial assumption about $p$, we know there exists $n<r$ such that $p\restriction W(\bar{s})\Vdash \dot{d}_l(\bar{s})=n$. By \ref{CL2}, $h_{W(\bar{s}), W(\bar{s}')}(p\restriction W(\bar{s}))\Vdash \dot{d}_l(\bar{s}')=n$. Hence $p'\Vdash \dot{d}_l(\bar{s}')=n=\dot{d}_l(\bar{s})$.

\end{enumerate}

\end{proof}

\begin{claim}\label{shrinking}
There exist $C^i\subset A_i$ containing $\alpha^i_\omega$ for $i<r$ such that
\begin{enumerate}
\item for each $i<r$, $type(C^i)=\omega+1$
\item for each $l\leq r$ and each index-strictly-increasing $l$-canonical tuple $$\bar{s}=(\alpha^0_{i_0}, \alpha^0_{i_0'}, \cdots, \alpha^{l-1}_{i_{l-1}}, \alpha^{l-1}_{i_{l-1}'}, \alpha^l_{i_l}, \alpha^{l+1}_{i_{l+1}}, \cdots, \alpha^{r-1}_{i_{r-1}})$$ from $\langle C^i: i<r\rangle$, 
$d_l(\bar{s})=d_l(\bar{s}')$, where 
$$\bar{s}'= \bar{s}_{\langle \alpha^j_{i_j'}: j<l\rangle^\frown \langle \alpha^j_{i_j}: r>j\geq l\rangle \to \langle \alpha^j_{\omega}: j<r\rangle}=(\alpha^0_{i_0}, \alpha^0_{\omega}, \cdots, \alpha^{l-1}_{i_{l-1}}, \alpha^{l-1}_{\omega}, \alpha^l_{\omega}, \alpha^{l+1}_{\omega},\cdots, \alpha^{r-1}_{\omega}).$$In particular, the color $\bar{s}$ gets under $d_l$ only depends on its index.

\end{enumerate}
\end{claim}

\begin{proof}
We will build these sets in $\omega$-steps. We will pick one point at a time from sets listed in the following order: $$A_0, A_1,\cdots, A_{r-1}, A_0, A_1,\cdots, A_{r-1}, A_0, A_1,\cdots, A_{r-1},\cdots.$$

In particular, we will find $J^i=\{j^i_k: k\in \omega\}\subset \omega$ such that $C^i=\{\alpha^i_{j^i_k}: k\in \omega\}\cup \{\alpha^i_\omega\}$ for each $i<r$. For fixed $i,k$, let $C^i_k$ denote $\{\alpha^i_{j^i_{k'}}: k'\leq k\} \cup \{\alpha^i_\omega\}$. 

Let $C_{-1}^i=\{\alpha^i_\omega\}$ for all $i<r$. Recursively, suppose for some $i<r$ and $k\in \omega$ we have defined $j^p_q$ for all $p<r, q \in \omega\cup \{-1\}$ and $\langle q,p\rangle <_{lex} \langle k,i\rangle$ (i.e. either $q<k$ or $q=k$ and $p<i$). Apply Claim \ref{ModularOperation} to pick $j^i_k\in \omega$ such that 
\begin{itemize}
\item $j^i_k > j^{p}_q$ for all $\langle q,p\rangle <_{lex} \langle k,i\rangle$
\item for any $l\leq r$ and any $l$-canonical tuple $\bar{s}$ containing $\alpha^i_\omega$ from $\langle C^p_{k_p}: p<r\rangle$ where $k_p=k$ if $p<i$ and $k_p=k-1$ if $p\geq i$, it is true that $d_l(\bar{s})=d_l(\bar{s}_{\alpha^i_\omega\to \alpha^i_{j^i_k}})$.
\end{itemize}

We now verify that $\langle C^i : i<r\rangle$ satisfies (2). Fix $l\leq r$ and some index-strictly-increasing $l$-canonical tuple $\bar{s}$ from $\langle C^i : i<r\rangle$, say $$\bar{s}=(\alpha^0_{i_0}, \alpha^0_{i_0'}, \cdots, \alpha^{l-1}_{i_{l-1}}, \alpha^{l-1}_{i_{l-1}'}, \alpha^l_{i_l}, \alpha^{l+1}_{i_{l+1}}, \cdots, \alpha^{r-1}_{i_{r-1}}).$$ By the hypothesis, we know $\max
\{i_m: m<r, i_m<\omega\}<i_k'$ for any $k<l$. By the conclusion of Claim \ref{ModularOperation} and the index management in our recursive process, we know that $$d_l(\bar{s})= d_l(\bar{s}_{\langle \alpha^j_{i_j'}: j<l\rangle^\frown \langle \alpha^j_{i_j}: r>j\geq l\rangle \to \langle \alpha^j_{\omega}: j<r\rangle}) $$$$ =d_l(\alpha^0_{i_0}, \alpha^0_{\omega}, \cdots, \alpha^{l-1}_{i_{l-1}}, \alpha^{l-1}_{\omega}, \alpha^l_{\omega}, \alpha^{l+1}_{\omega}, \cdots, \alpha^{r-1}_{\omega}).$$

\end{proof}

By Claim \ref{shrinking}, we may without loss of generality assume that the sets $\langle A_i: i<r\rangle$ already satisfy that: 
for each $l\leq r$, for each index-strictly-increasing $l$-canonical tuple $\bar{s}$ from $\langle A_i: i<r\rangle$ satisfies (2) in the conclusion of Claim \ref{shrinking}.

To finish the proof, we basically need similar arguments as in Claim 2.9 and Step 5 from \cite{6authors}. We supply a proof for completeness.
%
%

\begin{claim}\label{LastStepRamsey}
There exist $\langle B_i\subset A_i: i<r, \alpha^i_\omega\in B_i, type(B_i)=\omega+1\rangle$ and $\langle \rho_l<r: l\leq r\rangle$ such that for each $l\leq r$, for each index-strictly-increasing $l$-canonical tuple $\bar{s}$ from $\langle B_i: i<r\rangle$, $d_l(\bar{s})=\rho_l$.
\end{claim}
\begin{proof}
Fix $l\leq r, W\in [\omega]^{\aleph_0}$. Define $g: [W]^{l}\to r$ such that for each $\bar{i}=\{i_0<i_1<\cdots < i_{l-1}\}$,
 $$g(\bar{i})=d_l(\alpha^0_{i_0},\alpha^0_\omega, \cdots, \alpha^{l-1}_{i_{l-1}}, \alpha^{l-1}_\omega, \alpha^l_\omega, \alpha^{l+1}_\omega, \cdots, \alpha^{r-1}_\omega).$$ 
 
Let $I\in [W]^{\aleph_0}$ be a monochromatic subset with color $\rho_l$ for $g$. For any index-strictly-increasing $l$-canonical tuple $$\bar{s}=(\alpha^0_{j_0},\alpha^0_{j_0'}, \cdots, \alpha^{l-1}_{j_{l-1}}, \alpha^{l-1}_{j'_{l-1}}, \alpha^l_{j_l}, \alpha^{l+1}_{j_{l+1}}, \cdots, \alpha^{r-1}_{j_{r-1}})$$ such that $j_k, j'_t\in I\cup \{\omega\}$ for any $k<r$ and $t<l$,
 by Claim \ref{shrinking} and the remark that follows, we know that $$d_l(\bar{s})=d_l(\alpha^0_{j_0},\alpha^0_{\omega}, \cdots, \alpha^{l-1}_{j_{l-1}}, \alpha^{l-1}_{\omega}, \alpha^l_{\omega}, \alpha^{l+1}_{\omega}, \cdots, \alpha^{r-1}_{\omega})=g(\{j_0<\cdots < j_{l-1}\})=\rho_l.$$
 
To get the conclusion of the claim, apply the procedure above repeatedly to get $\omega \supset I_0\supset I_1\supset \cdots \supset I_{r-1}\supset I_r$.
It is clear that $B_i=\{\alpha^i_j: j\in I_{r}\}\cup \{\alpha^i_\omega\}$ for $i<r$ will be the desired sets.

\end{proof}

%
%
%
%

By Claim \ref{LastStepRamsey}, we may without loss of generality assume that the sets $\langle A_i: i<r\rangle$ already satisfy that: there exist $\langle \rho_l: l\leq r\rangle$ such that for each $l\leq r$, for each index-strictly-increasing $l$-canonical tuple $\bar{s}$ from $\langle A_i: i<r\rangle$, $d_l(\bar{s})=\rho_l$. By the Pigeon hole principle, there exist $l'<l$ such that $\rho_{l'}=\rho_l=\rho$. 

\begin{claim}\label{finish(1)}
There exists an infinite $X$ such that $f\restriction X+X \equiv \rho$.
\end{claim}

\begin{proof}
For $i<\omega$, let 
\begin{equation*}
\begin{split}
\bar{a}_i=(\alpha^0_0, \alpha^0_\omega,\alpha^1_1, \alpha^1_\omega, \cdots, \alpha^{l'-1}_{l'-1}, \alpha^{l'-1}_\omega,\\ \alpha^{l'}_{l'+i\cdot (l-l')},\alpha^{l'+1}_{l'+1+i\cdot (l-l')},\cdots,  \alpha^{l-1}_{l-1+i\cdot (l-l')}, \\ \alpha^{l}_{\omega}, \cdots, \alpha^{r-1}_{\omega}),
\end{split}
\end{equation*}

namely, we take 

\begin{enumerate}

\item $\{\alpha^k_{k}, \alpha^k_\omega\}$ from $A_k$ for each $k<l'$
\item $\{\alpha^k_{k+i(l-l')}\}$ from $A_k$ for each $k\geq l'$ and $k<l$
\item $\{\alpha^k_{\omega}\}$ from $A_k$ for each $k\geq l$.
\end{enumerate}

Define $x_i=\dfrac{1}{2} s_{l'}* \bar{a}_i$. For $i < j \in \omega$, consider
\begin{equation*}
\begin{split}
\bar{b}_{i,j}=(\alpha^0_0, \alpha^0_\omega,\cdots, \alpha^{l'-1}_{l'-1}, \alpha^{l'-1}_\omega, \\ \alpha^{l'}_{l'+i\cdot (l-l')}, \alpha^{l'}_{l'+j\cdot (l-l')}, \cdots \alpha^{l-1}_{l-1+i\cdot (l-l')}, \alpha^{l-1}_{l-1+j\cdot (l-l')}, \\ \alpha^{l}_{\omega}, \cdots, \alpha^{r-1}_{\omega}),
\end{split}
\end{equation*}

namely, we take 

\begin{enumerate}

\item $\{\alpha^k_{k}, \alpha^k_\omega\}$ from $A_k$ for each $k<l'$
\item $\{\alpha^k_{k+i(l-l')}, \alpha^k_{k+j(l-l')}\}$ from $A_k$ for each $k\geq l'$ and $k<l$
\item $\{\alpha^k_{\omega}\}$ from $A_k$ for each $k\geq l$.
\end{enumerate}
It is not hard to notice that $x_i+x_j = s_l* \bar{b}_{i,j}$. To check that $\bar{b}_{i,j}$ is an $l$-canonical tuple, let $k<l$ be given. If $k<l'$, then it is immediate that $k<\omega$. If $l'\leq k<l$, then $k+i(l-l') < k+(i+1)(l-l')\leq k+j(l-l')$ since $l-l'>0$. Also notice that $l'+j(l-l') \geq l'+(i+1)(l-l') =l + i (l-l') > l-1 + i(l-l')$. 

Therefore, for any $i<j\in \omega$, $\bar{a}_i$ ($\bar{b}_{i,j}$ respectively) is easily seen to be an index-strictly-increasing $l'$-canonical ($l$-canonical) tuple. Therefore, $f(2x_i)=f(s_{l'}*\bar{a}_i)=d_{l'}(\bar{a}_i)=\rho_{l'}=\rho$ and $f(x_i+x_j)=f(s_l*\bar{b}_{i,j})=d_l(\bar{b}_{i,j})=\rho_l =\rho$. We conclude that $X=\{x_i: i\in \omega\}$ is the set as desired.

\end{proof}
Claim \ref{finish(1)} finishes the proof of (\ref{part1}).

%
%
%
%
%
\section{Acknowledgement}
I thank Dániel Soukup for telling me about this problem and sharing his insights. I thank the organizers and the participants of the SETTOP 2018 conference in Novi Sad during which part of this work was done. I thank the Department of Mathematical Sciences of Carnegie Mellon University for the financial support for me to attend SETTOP 2018. 

I thank James Cummings, Sittinon Jirattikansakul and Stevo Todor\v{c}evi\'{c} for helpful discussions and comments.

Finally, I am grateful to the anonymous referee who provides extensive comments and corrections which greatly improve the exposition.

\bibliographystyle{plain}
\bibliography{bib}

\Addresses

\end{document}